\documentclass[10pt,a4paper,draft]{article}

\usepackage{amsmath,amssymb,amsthm,amscd}
\usepackage[]{fontenc}
\usepackage{xy}
\usepackage{enumerate}

\xyoption{all} \numberwithin{equation}{section}
\newtheorem{theorem}[equation]{Theorem}

\newtheorem{proposition}[equation]{Proposition}
\newtheorem{lemma}[equation]{Lemma}

\theoremstyle{definition}

\newtheorem{definition}[equation]{Definition}

\newtheorem{example}[equation]{Example}

\newtheorem{remark}[equation]{Remark}

\newcommand{\nc}{\newcommand}
\nc{\cH}{\mathcal{H}} \nc{\cA}{\mathcal{A}} \nc{\cG}{\mathcal{G}}
\nc{\cC}{\mathcal{C}}
\nc{\cO}{\mathcal{O}}
\nc{\cI}{\mathcal{I}}
\nc{\cB}{\mathcal{B}} \nc{\cY}{\mathcal{Y}} \nc{\cK}{\mathcal{K}}
\nc{\cX}{\mathcal{X}} \nc{\cS}{\mathcal{S}} \nc{\cE}{\mathcal{E}}
\nc{\cF}{\mathcal{F}} \nc{\cZ}{\mathcal{Z}} \nc{\cQ}{\mathcal{Q}}
\nc{\cN}{\mathcal{N}} \nc{\cP}{\mathcal{P}} \nc{\cL}{\mathcal{L}}
\nc{\cM}{\mathcal{M}} \nc{\cR}{\mathcal{R}} \nc{\cT}{\mathcal{T}}
\nc{\cW}{\mathcal{W}} \nc{\cU}{\mathcal{U}} \nc{\cD}{\mathcal{D}}
\nc{\cJ}{\mathcal{J}} \nc{\cV}{\mathcal{V}}
\nc{\fr}{{\rightarrow}}
\nc{\rd}{red.deg}

\title{Generalized Clifford-Severi Inequality  and \break the Volume of Irregular Varieties}
\author{Miguel A. Barja \footnote{Partially supported by MICINN-MTM2009-14163-C02-02/FEDER, MINECO-MTM2012-38122-C03-01 and by Generalitat de Catalunya 2005SGR00557.}}

\begin{document}

\maketitle
%{\centerline {\small Departamet de Matem\`atica Apliada 1. Universitat Polit\`ecnica de Catalunya-BarcelonaTECH}}
%{\centerline {\small Avda. Diagonal 647 08028-Barcelona Spain}}
%{\centerline {\small miguel.angel.barja@upc.edu}}

\begin{abstract}  We give a sharp lower bound for the self-intersection of a nef line bundle $L$ on an irregular variety $X$ in terms of its continuous global sections and the Albanese dimension of $X$, which we call the Generalized Clifford-Severi inequality. We also extend the result to nef vector bundles and give a slope inequality for fibred irregular varieties. As a byproduct we obtain a lower bound for the volume of irregular varieties; when $X$ is of maximal Albanese dimension the bound is ${\rm vol}(X) \geq 2 \, n! \, {\chi}(\omega _X)$ and it is sharp.
\end{abstract}

\bigskip
\bigskip
\bigskip

\section{Introduction and preliminaries}

The geometry of irregular varieties has been deeply developed in the last 30 years. The seminal results of Green and Lazarsfeld on Generic Vanishing theorems, the generalized Castelnuvo-de Franchis theorem by Catanese and Ran, the systematic use of the Fourier-Mukai techniques associated to the Albanese map and further developments by Lazarsfeld, Ein, Hacon, Chen, Pareschi, Popa and many others have provided a rather complete understanding of birational properties of such varieties.

%Although we do not follow this approach, we will use some definitions and we will refer to \cite{Pareschi} and \cite{Pareschipopa} for references.

Another fruitful approach to understanding the geometry of an irregular variety is the study of their {\it continuous} linear series. In the case of abelian varieties this approach goes back to Mumford and Kempf (see for example \cite{M1}, \cite{M2}). For the study of irregular surfaces, one of the first instances of a systematic use of them is done in  \cite{CCML}. Later on the work of Pareschi and Popa on continuous global generation and the recent work of Mendes-Lopes, Pardini and Pirola on Brill-Noether theory and continuous families of divisors gives a deep understanding of their geometry in higher dimensions (see \cite{MLPSurvey}, \cite{MLPP2}, and  \cite{MLPP1}).

In the study of biregular geometry of varieties, inequalities relating the degree and the number of global sections of a line bundle play a special role. This is the goal of our main theorem. In order to stating it, we need to introduce some notation. Consider an irregular variety $X$ and a non trivial map $a:X \longrightarrow A$ to an abelian variety. We say that $X$ is of {\it maximal $a$-dimension} if ${\rm dim} \, a(X)={\rm dim} \, X$. We associate to any nef line bundle $L$ on $X$ two invariants. One of them is $\delta(L)=\frac{2r(L)}{2r(L)-1}$ (see Definition \ref{delta1}), a real number between 2 an 1, where $r(L)$ is the degree of subcanonicity of $L$, the minimal $r$ such that $L\preceq rK_X$ (numerically). The second one is $h^0_a(L)$, the {\it continuous rank} of $L$, i.e., the rank of the Fourier-Mukai transform of $L$ with respect to the map $a$, i.e. the minimal value of $h^0(X,L\otimes a^*\alpha)$ for $\alpha \in {\widehat A}$. We also will consider $W$ and $M=L-W$, the {\it continuous fixed part} and the {\it continuous moving part} of $L$, respectively, where $W$ is the common base component of the linear systems $|L\otimes a^*\alpha|$, for $\alpha$ general (see Definition \ref{continuousmovingpart}). With this notation, we can state

\bigskip

\noindent {\large {\bf Main Theorem (Generalized Clifford-Severi Inequality)}}{\it
\smallskip

Let $X$ be a smooth, projective variety of dimension $n$, over an algebraically closed field of characteristic 0. Let $a: X \longrightarrow A$ be a nontrivial map to an Abelian variety and let $L \in {\rm Pic}(X)$ be a nef line bundle.
\begin{itemize}

\item [(i)] If $X$ is of maximal $a$-dimension then $$L^n \geq  \delta (L)\, n! \,h^0_a(L).$$
\noindent In particular, if $L \preceq K_X$, then $L^n \geq 2n! \, h^0_a(L).$
\item [(ii)] Assume $n> {\rm dim} \, a(X)=k\geq 1$ and let $M$ be the continuous moving part of $L$. If $M$ is $a$-big then
$$L^n \geq  \delta (L)\, k! \,h^0_a(L).$$
\item [(iii)] Assume that $n>{\rm dim} \, a(X)=k\geq 1$ and that $L$ is $a$-big. Then  $$L^n \geq  k! \,h^0_a(L).$$
\end{itemize}
}

\bigskip

We can see this theorem as a wide generalization of the classical Severi inequality for surfaces of maximal Albanese dimension:

$$K^2_S\geq 4 \,\chi(\omega_S).$$

This inequality was stated by Severi in the 30's (\cite{Severi}). Many years later, Catanese (\cite{Catanese}) showed a gap in the proof and proposed the inequality as a conjecture. Manetti gave a proof of the conjecture in the case $K_S$ ample, together with a profound analysis of the positivity properties of $\Omega^1_S$ (\cite{Manetti}). His approach provides further developments and refinements (see \cite{Mendeslopespardini} and \cite{Zhang}).

The key argument for a complete proof of the Severi inequality without extra hypotheses is given by Pardini (\cite{Pardini}, \cite{MLPSurvey}). She deduces the inequality from another well known one: the slope inequality for fibred surfaces (\cite{Xiao}, \cite{CH}), by using in a quick and clever way the property of being of maximal Albanese dimension. This method, which we call {\it Pardini's covering trick}, is completely general, and allows to apply a general principle: given an inequality verified by any maximal Albanese dimension variety you can remove all the numerical data involving lower dimensional subvarieties and obtain a new inequality.

%This strategy can not be generalized using higher dimensional slope inequalities since no one is known to hold, not even for the canonical sheaf. These slope %inequalities would follow from stability conditions which are not known to hold in higher dimensions (see \cite{survey} section 5 for a complete discussion and %some results in this direction).

In the present paper  we generalize the Severi inequality by an induction argument on the dimension of $X$. For this, we combine three different ingredients:

\begin{itemize}
\item A suitable version of Xiao's method for fibrations reducing the problem to \'etale covers and lower dimensional varieties (induction step). This is done in subsection \ref{subsectionxiao}.
\item The analysis of the behavior of {\it continuous} linear series on $X$, i.e., $|L\otimes \alpha|$ for $\alpha \in {\widehat A}$. It allows to assume good behavior of the linear system on an \'etale covering of $X$. This is done in Section 3.
\item The use of Pardini's covering trick (\cite{Pardini}) to remove unnecessary invariants. This is done in subsection \ref{prueba}.
\end{itemize}

The initial step of the induction process (the case of curves) is just a continuous version of Clifford's Lemma which we can also consider as a 1-dimensional version of Severi inequality. That's the reason why we add Clifford in the name of the inequality.

There are several particular cases of the main result of independent interest, which are introduced in Chapter 4. Here we present the more relevant ones. The first one is an extension to vector bundles

\medskip

\noindent {\bf Corollary A} ({\bf Generalized Clifford-Severi inequality for nef Vector Bundles}) {\it
Let $X$ be a projective, smooth variety of dimension $n$, over an algebraically closed field of characteristic 0. Let $a: X \longrightarrow A$ be a nontrivial map to an Abelian variety such that ${\rm dim}\,a(X)=k$. Let ${\cal F}$ be a {\it nef} vector bundle on $X$ with top Segre class $s({\cal F})$. Assume that $k=n$ or that ${\cal F}$ is $a$-big. Then

$$s({\cal F})\geq k! \, h^0_a({\cal F}).$$
}
\bigskip

%Secondly, when $L=K_X$ we obtain a generalization of classical Severi inequality, when $X$ is minimal. In Chapter 4 we give a proof adapted for $Q$-factorial %varieties with canonical singularities. Most results can be extended to Weil divisors on normal $\mathbb{Q}$-Gorenstein varieties just by understanding the %behavior under the desingularization process. Some of the techniques are presented in general for Weil divisors, as Xiao's method in Subsection %\ref{subsectionxiao}, and in Corollary \ref{Severi} we reinterpret the Generalized Clifford-Severi inequality as a lower bound for the volume of an irregular %variety. The (canonical) volume of a general type variety is an interesting invariant of its birational class. In general it can be very small and is not easy %to find higher lower bounds (see for example \cite{Ch}). In the presence of global differential forms it was conjectured by Reid that the value of ${\rm %vol}(X)$ should have a high lower bound. As far as we know even in the case of irregular 3-folds there are few known results (see for example \cite{ChCh}, %section 3). The bound we give here confirms Reid's conjecture for 1-forms. The best bound is obtained in the case of maximal Albanese dimension varieties of %general type, for which it is sharp (double covers of abelian varieties):

%$${\rm vol}(X)\geq 2\, n! \, \chi(\omega_X).$$
When $L=K_X$ we obtain the sharp generalization of the classical Severi inequality or, equivalently, a lower bound for the volume of irregular varieties.

\medskip

\noindent {\bf Corollary B} ({\bf Generalized Severi inequality for $K_X$/the volume of irregular varieties}) {\it
 Let $X$ be an irregular, minimal, normal, projective, $\mathbb{Q}$-factorial variety of dimension $n$ over an algebraically closed field of characteristic 0. Let $X'$ be any desingularization of $X$. Then
\begin{itemize}
\item [(i)] If $X$ is of maximal Albanese dimension then
$$K_X^n \geq 2n!\,{\chi}(\omega _{X'})$$

\noindent and this bound is sharp (double covers of abelian varieties).
\item [(ii)]  If ${\rm dim} \, alb_X(X)=k<n$ and the continuous moving part $M$ of $K_{X'}$ is ${\rm alb}_{X'}$-big, then

$$K_X^n \geq 2k!\,h^0_{alb_{X'}}(\omega _{X'}).$$

\item [(iii)] If ${\rm dim} \, alb_X(X)=k<n$ and $X$ is of general type, then

$$K_X^n \geq k!\,h^0_{alb_{X'}}(\omega _{X'}).$$

\end{itemize}
}

\bigskip

This is not a merely application of Main Theorem but we must translate the problem of working with the $\mathbb{Q}$-Cartier Weil divisor $K_X$ on the singular variety $X$ to a suitable line bundle $L$ on a desingularization $X'$.

We can understand the previous result as a lower bound for the canonical volume of any smooth irregular variety $X'$ of general type, just applying Corollary B to a minimal model. It is not easy to find explicit lower bounds for the canonical volume of general type varieties (see for example \cite{Ch}) and in general it can be very small. It was conjectured by Reid that for varieties with global differential forms the volume should be high. As far as we know, even for the case of irregular 3-folds few results are known (see for example \cite{ChCh}, section 3). The bound given by Corollary B gives a sharp lower bound in the maximal Albanese dimension case, attained by double covers of abelian varieties:

$${\rm vol}(X')\geq 2\,n!\,{\chi}(\omega_{X'}).$$

During the final preparation of this work, the author has been informed that the content of Corollary B (i), in the case of minimal Gorenstein varieties, was independently proven by Zhang (\cite{TongZhang}). There the strategy of proof relies on applying Pardini's trick to the higher dimensional version of the so called Relative Noether inequality, bounding the linear sections of a nef line bundle on an fibred variety.

\medskip

%As another instance of the Main Theorem we have
%\medskip

%\noindent {\bf Corollary C} {\bf  (Clifford-Severi inequality for adjoint divisors)}
%{\it Assume that $X$ is a smooth $n$-dimensional variety of maximal Albanese dimension. If $L=K_X+D$, with $D$ nef, then

%$$L^n \geq \delta(L)\,n!\,{\chi}(X,L).$$
%}

%\item[(ii)] If $L=K_X+D$, with $D$ big and nef, then $L^n \geq \delta (L)\,n!\,h^0(X,L)$.
%\end{itemize}

%\medskip
%For the classification of irregular surfaces with small invariants the following is an interesting result

%\medskip

%\noindent {\bf Corollary D} {\bf (Decomposable $K_S$ on surfaces)} {\it Let $S$ be a smooth surface of maximal Albanese dimension and such that %$K_S\equiv L_1 + L_2$ (numerically), with $L_i$ nef line bundles. Then

%$$K^2_S \geq 4{\chi}(\omega _S) + 4 h^1_a(L_1).$$
%}
%\medskip

Our last application shows as the Clifford-Severi inequality implies a slope inequality for fibred varieties. Given a fibred variety over a smooth curve, $f: X \longrightarrow B$ with general fibre $F$, and a line bundle $L$ on $X$, slope inequalities relate the invariants of $(X,L)$, $(F,L_{|F})$ and $B$. The best slope inequalities hold when some stability properties hold for $L$ (see \cite{survey} for a survey on this topic), and then a slope inequality, which we call $f$-{\it positivity of }$L$, follows

$$L^n\geq n\,\frac{L_{|F}^{n-1}}{h^0(F,L_{|F})}\,{\rm deg}f_*L.$$

\noindent The case $X$ of general type and $L=\omega_f$ encodes important numerical and geographical properties of the fibration, but $f$-positivity of $\omega_f$ is only known for fibred surfaces. In fact, $f$-positivity of $\omega_f$  for dimensions less or equal to $n$  implies a weaker inequality for the slope

$$K^n_f \geq 2n!\,\chi_f$$

\noindent from which one can deduce the Severi inequality for $L=\omega_X$ (see \cite{survey} Proposition 5.8, for a proof of this statement). We prove a converse, namely, that the Severi inequality for $L=\omega_f$ implies the above slope inequality when $b=g(B)=0$ and a slightly weaker result when $b \geq 1$.

\medskip

\noindent {\bf Corollary C} ({\bf Slope inequality}) {\it Let  $f:X \longrightarrow B$ be a relatively minimal fibration onto a smooth curve $B$ of genus $b$ with general fibre $F$ and ${\rm dim} \, X=n$. Assume that $X$ is of maximal Albanese dimension. Then

\begin{itemize}
\item [(i)] If $b=0$ then $K_f^n\geq 2 n! \,\chi_f$.
\item [(ii)] If $b \geq 1$, then $K_f^n \geq 2n!\,[{\chi}(\omega _f)+h^1_a(\omega _f)] \geq 2n!\,({\chi}(\omega _X) - 2{\chi} (\omega _B){\chi} (\omega _F)).$
\end{itemize}
}
\noindent
\medskip

The paper is divided as follows. The general set-up is  $X$ a smooth, projective variety over an algebraically closed field of characteristic 0, with a nontrivial map $a:X \longrightarrow A$ to an abelian variety.

In section 2 we study the properties of the invariants $h^0_a(L)$ and $\delta(L)$ associated to any line bundle $L$ on $X$. Section 3 is devoted to studying the behavior of linear systems under \'etale Galois covers. Here, the concepts of {\it continuous fixed part} and {\it continuous moving part} of $L$ play an special role. This analysis is a cornerstone in the proof of the main theorem and provides results of independent interest on linear series on irregular varieties.  It turns out that (see Theorems \ref{continuousresolution} and \ref{etale} for more complete results):

%The first one is its {\it continuous rank} $h^0_a(L)$ (see section \ref{continuousrank}) which is just the minimum value of $h^0(X,L\otimes a^*\alpha)$, for $\alpha \in {\widehat A}$. For instance, the Generic Vanishing theorem shows that $h^0_a(\omega_X)=\chi(\omega_X)$ for varieties of maximal Albanese dimension. In the same vein we can define $h^i_a(L)$ for any $i \geq 1$. We also associate to $L$ its degree of subcanonicity (the minimal $r$ such that $L \preceq rK_X$) and a real number $2\geq \delta (L) \geq 1$ defined in terms of $r$ (see Definition \ref{delta1} and Remark \ref{delta2}), in such a way that $\delta(L)=2$ if and only if $L$ is numerically subcanonical ($r=1$).

\medskip

\noindent {\bf Theorem D}  {\it Let $X$ be a smooth, projective variety and let $a: X \longrightarrow A$ be a nontrivial map to an Abelian variety such that $a^*: {\widehat A}\longrightarrow {\rm Pic}^0X$ is injective. Up to composing a blow-up with an \'etale Galois covering, $\lambda: {\widetilde X} \longrightarrow X$
\begin{itemize}
\item [(i)] For any $\alpha \in {\widehat A}$ we have a decomposition: $\lambda^*(L\otimes \alpha)={\widetilde W}+{\widetilde N_{\alpha}}$ where the divisor ${\widetilde W}$ is the fixed component (and does not depend on $\alpha$) and the moving part ${\widetilde N_{\alpha}}$ is base point free.
\item [(ii)] The map $a\circ \lambda$ factors through the algebraic fibre space induced by the linear system $|\lambda^*L|$. In particular ${\rm dim} \, \phi_{\lambda^*L}({\widetilde X})\geq {\rm dim} \, a(X)$ and so $|\lambda^*L|$ is generically finite provided $X$ is of maximal $a$-dimension.
\end{itemize}}

\medskip

Section 4 is devoted to proving Corollaries A, B and C, and giving some remarks, examples and other applications.
In Section 5 we prove the Main Theorem. There, an account of Xiao's method especially adapted to \'etale Galois covers of fibrations onto $\mathbb{P}^1$ is given.

\bigskip
\bigskip

\noindent {\bf Notations and conventions} Varieties are assumed to be smooth, projective, defined over an algebraically closed field $k$ of characteristic 0, except otherwise stated.
We use the notation $L$ for a (Cartier) divisor or its associated line bundle interchangeably, except for the canonical sheaf and divisor which will be denoted by $\omega_X$ and $K_X$ respectively.  We use additive or multiplicative notation interchangeably.

Given an abelian variety $A$ we denote by ${\widehat A}={\rm Pic}^0(A)$ its dual abelian variety, by ${\widehat A}_d$ the subgroup of its $d$-torsion elements and by ${\widehat A}_{{\rm tors}}=\bigcup_{d \in \mathbb{N}}{\widehat A}_d$ the set of all its torsion elements.

Given an irregular variety $X$ we set ${\rm Pic}^{\tau}(X)$ for the set of numerically torsion line bundles on $X$, i.e., the set of $M$ such that $M^{\otimes r}\in {\rm Pic}^0(X)$ for some $r \in \mathbb{N}$.
%
%Maps $a:X \longrightarrow A$, where $A$ is an Abelian variety are assumed to be {\it generating}, i.e., that $a(X)$ generates $A$ as a %group.

Given a map $f: X \longrightarrow Y$ (not necessarily surjective) we say that it is {\it generically finite} if $X \longrightarrow f(X)$ is. We say that $f$ factors through an algebraic fiber space of dimension $k$ with general fibre $G$ if dim$f(X)=k$ and so its Stein factorization decomposes $f=g\circ h: X \longrightarrow Z \longrightarrow Y$ where $Z$ is normal of dimension $k$, $g$ is an algebraic fiber space and $G$ is a general fibre of $g$. As usual, an algebraic fiber space of dimension 1 will be called a {\it fibration}.

We will use $\equiv$ for numerical equivalence and given two divisors $D_1, D_2$ we denote $D_1\preceq D_2$ if $D_2-D_1$ is pseudoeffective, i.e., it is a limit of (real) effective divisors or, equivalently, its product with arbitrary nef line bundles is nonnegative.

\bigskip

\noindent{\bf Acknowledgements}
The author wants to thank Rita Pardini, Lidia Stoppino, Gian Pietro Pirola, Joan Carles Naranjo, Margarida Mendes-Lopes and Mart{\'\i} Lahoz for fruitful conversations and encouragement during the preparation of this work. The author also wants to thank the anonymous referee for several suggestions to improve the presentation and proofs and for pointing out several inaccuracies in the first version of the paper.

%%%%%%%%%%%%%%%%%%%%%%%%%%%%%%%%%%%%%%%%%%%%%%%%%%%%%%%%%%%%%%%%%%%

\section{The continuous rank and the subcanonicity index}\label{continuousrank}

In this section we introduce the {\it continuous rank} of any coherent sheaf on irregular varieties, and the {\it subcanonicity degree} of a line bundle on a variety. The fundamental properties of these two numbers are that both behave well under algebraic equivalence, blow-ups, \'etale Galois coverings and hyperplane sections.

Given a line bundle on an irregular variety $X$, the behavior of the continuous system given by $|L\otimes \alpha|$ is studied in \cite{MLPP1} and in \cite{MLPP2}. Our interest relies on {\it general} elements in the continuous family.

\begin{definition} Let $X$ be an irregular variety and ${\cal F}$ a coherent sheaf on $X$. Let $A$ be an abelian variety and $a: X \longrightarrow A$ a nontrivial map. Define

$$h^i_a({\cal F}):={\rm min}\{h^i(X,{\cal F}\otimes a^*\alpha) \, | \, \alpha \in {\rm Pic}^0(A) \, \}$$

\noindent $h^0_a({\cal F})$ will be called the {\it continuous rank} of ${\cal F}$.

\smallskip

%\noindent We can extend the definition to singular irregular varieties just considering .....

\noindent By abuse of notation, if the map is clear by the context, we will usually write $\alpha$ instead of $a^*\alpha$.
\end{definition}

\begin{remark} If $RS_a$ is the integral Fourier-Mukai functor associated to the map $a$ (see \cite{Pareschi}), we have that $h^i_a({\cal F})={\rm rank}RS^i_a({\cal F})$.
\end{remark}

\begin{definition} If $a: X \longrightarrow A$ is a map from $X$ to an abelian variety, we will say that $X$ {\it is of maximal a-dimension} if dim$X$
=dim$\,a(X)$. When $A={\rm Alb}(X)$ and $a={\rm alb}_X$ the definition corresponds to the classical notion of a maximal Albanese dimension variety.
\end{definition}

\begin{example}\label{ejemplos} As simple examples related to line bundles $L \in {\rm Pic}(X)$ on smooth $X$, we have

\begin{itemize}
\item If $L=K_X+D$ with $D$ nef, and $X$ is of maximal $a$-dimension, then $h^0_a(L)=\chi(X,L)$ by the Generic Vanishing theorem given by Pareschi-Popa (Theorem B in \cite{Pareschipopa}). If ${\rm dim} \, alb(X)={\rm dim} \, X -1$ then we still have $h^0_a({\cal L})\geq h^0_a({\cal L})-h^1_a({\cal L})=\chi(X,{\cal L})$.

\item The same holds for any $GV$-sheaf, like higher direct images of relative dualizing sheaves (see also \cite{Pareschipopa}).

\item If $L=K_X+D$ with $D$ big and nef, then $h^0_a(L)=\chi(X,L)=h^0(X,L)$ by Kawamata-Viehweg vanishing theorem.
%\item If $Z$ is a common component of all the linear systems $|L\otimes a^*(\alpha)|$, then $h^0_a(L)=h^0_a(L-Z)$
\end{itemize}
\end{example}

\begin{remark} The author was informed by Pirola that using the arguments of \cite{MLPP1} and \cite{MLPP2} and bounding the obstructions to deforming a global section of $L$, the following inequality holds: $h^0_a(L)\geq h^0(X,L)-h^1(X,L)$ (unpublished).

\end{remark}

\begin{remark} The Severi inequality can be restated as $K^2_S \geq 4 h^0
_{alb_S}(\omega_S)$.
\end{remark}

Let's see now a non trivial example

\begin{proposition} \label{h0slope} Let $f: X \longrightarrow B$ be a fibration onto a smooth, projective curve of genus $b$, with general fibre F. Assume that $X$ is smooth of general type and of maximal Albanese dimension. Let $a
=alb_X$.Then

$$ h^0_a(\omega_f)=h^1_a(\omega_f)+\chi(X,\omega_f)=h^1_a(\omega_f)+\chi(X,\omega_X)-2\chi(B,\omega_B)\chi(F,\omega_F).$$

\noindent If $b=0,1$ we have in fact that $h^0_a(\omega_f)=\chi(X,\omega_X)-(2b-2)\chi(F,\omega_F)\geq \chi(X,\omega_X).$

\end{proposition}

\begin{proof} Let $A=Alb(X)$. The subset of ${\rm Pic}^0(A)$ where $h^0(X,\omega_f \otimes a^*\alpha)$ takes its minimum value is an open set, so it is enough to compute this value for a general  $\alpha \in {\widehat A}
_{{\rm tors}}$. In this case $R^if_*(\omega_f \otimes \alpha)$ is locally free on $B$ (\cite{Haconpardini}), of rank $h^i(F,\omega_F \otimes \alpha)$, which is zero for $i \geq 1$ by generic vanishing applied to the map ${\overline a}: F \hookrightarrow X \longrightarrow A$ (clearly $F$ is of maximal ${\overline a}$-dimension). Hence $H^i(X, \omega_f \otimes \alpha)=H^i(B,f_*(\omega_f \otimes \alpha))=0$ for $i\geq 2$, and the result follows. Observe that the equality $\chi(X,\omega_f)=\chi(X,\omega_X)-2\chi(B,\omega_B)\chi(F,\omega_F)$ holds since for any line bundle on $X$ the equality $\chi(X,L+kF)=\chi(X,L)+k\chi(F,L_{|F})$ follows by induction on $k$.

When $b=0,1$, consider the \'etale covering given by $\alpha$, $g: {\overline X}\longrightarrow X$. We can choose $\alpha$ of prime order $p$ such that ${\overline a}^*\alpha \neq {\cal O}_F$, and so ${\overline a}^*\alpha ^{\otimes i}\neq {\cal O}_F$ for all $1\leq i\leq p-1$.
 %Moreover, we can choose it such that $\alpha \notin V^1_a({\cal O}_X)$.
Hence, the induced map ${\overline f}=f \circ g: {\overline X} \longrightarrow B$ is a fibration (the fibres are the connected \'etale cover of $F$ given by ${\overline a}^*\alpha \neq {\cal O}_F$).
%$q({\overline X})=q(X)$.
Then $f_*(\omega_f \otimes \alpha)$ is a direct summand of  ${\overline f}_*(\omega_{\overline f})$, and so it is a nef vector bundle on $B$, since  ${\overline f}_*(\omega_{\overline f})$ is (\cite{fujita}). If $b=0$ this proves $H^1(X, \omega_f \otimes \alpha)=H^1(B,f_*(\omega_f \otimes \alpha))=0$. If $b=1$,
 %by the decomposition theorem of Fujita, since $q({\overline X})=q(X)$,
we also obtain that $H^1(X, \omega_f \otimes \alpha)=H^1(B,\omega_X \otimes \alpha)=0$ by generic vanishing.
%H^0(B,f_*(\omega_f \otimes \alpha)^*)=0$.
\end{proof}

Finally let us see how $h^i_a$ behaves under \'etale covers, blow-ups and hyperplane sections.

\begin{proposition}\label{etale} Let $X$ be a smooth, projective irregular variety and let $a: X \longrightarrow A$ be a nontrivial map to an abelian variety $A$ such that $a^*:{\widehat A} \longrightarrow {\rm Pic}^0X$ is injective. Take $L \in {\rm Pic}(X)$.
\begin{itemize}
\item [(i)] Let $\mu : \widetilde{A} \longrightarrow A$ a degree $m$ isogeny and consider the base change diagram

$$\xymatrix { {\tilde X} \ar[r] ^{\tilde \mu} \ar[d] ^{\tilde a} & X \ar[d] ^a \\ {\tilde A} \ar[r] ^{\mu} & A }$$

Then $h^i_{\tilde a}({\tilde \mu} ^*(L)=mh^i_a(L)$.

\item [(ii)] If $\sigma: X' \longrightarrow X$ is a blow-up, then $h^0_a(\sigma^*L)=h^0_a(L)$.
\item [(iii)] If $H$ is a smooth divisor on $X$ not contracted by $a$, such that $H-L$ is nef, then $h^0_a(L_{|H})\geq h^0_a(L)$.
\end{itemize}
\end{proposition}

\begin{proof}

\noindent (i) Let $N={\rm Ker}\mu \subseteq {\hat A}$, which is of order $m$. We have that ${\widetilde X}$ is connected since $a^*$ is injective and ${\tilde \mu} _*({\cal O}_{{\tilde X}})= \bigoplus_{\gamma \in N}a^*(\gamma)$. Let $Z_i\subsetneq {\hat A}$ be the jumping locus of the value $h^i(X, a^*(\alpha))$, which is a proper closed set. For any $\beta =\mu ^*(\alpha) \notin \mu ^* (Z_i)$ we have

$$h^i({\tilde X}, {\tilde \mu}^*L \otimes {\tilde a}^*\beta)=\bigoplus _{\gamma \in N}h^i(X,L \otimes a^*(\alpha \otimes \gamma))=mh^i_a(L)$$

\noindent since $\alpha \otimes \gamma \notin Z_i$ for all $\gamma \in N$.

\noindent (ii) Obvious.

\noindent (iii) For $\alpha \in {\widehat A}$ we can consider the exact sequence

$$0\longrightarrow {\cal O}_X((L-H)\otimes \alpha)\longrightarrow {\cal O}_X(L\otimes \alpha)\longrightarrow {\cal O}_H(L_H\otimes \alpha)\longrightarrow 0$$

Taking cohomology we have $h^0(X,(L-H)\otimes \alpha)=h^n(X,K_X+(H-L)\otimes \alpha^{-1})=0$ for $\alpha$ general, by the generic vanishing theorem (cf. Theorem B in \cite{Pareschipopa}), since we are asuming $H-L$ is nef.

\end{proof}

Let us consider now the second invariant we need. As in the case of curves, the smaller is the degree of a nef line bundle, the bigger is the ratio between its degree and its global sections. The best behavior, given by Clifford's lemma, holds for subcanonical line bundles. We introduce now an invariant which measures exactly this relation in terms of $r$-subcanonicity. More concretely we have

\begin{definition}\label{delta1} Given a nef line bundle $L$ we define

\begin{itemize}
\item [(i)] $r(L)= {\rm inf}\{r \geq 1 \,|\,L \preceq rK_X \}$ (degree of subcanonicity of $L$). Note that if this set is empty, then $r(L)=\infty$. We will say that $L$ is {\it subcanonical} if $r(L)=1$.
\item[(ii)] $\delta (L) =\frac{2r(L)}{2r(L)-1}$.
\end{itemize}
\end{definition}

\begin{remark}\label{delta2} The following are easy properties of $\delta$:

\begin{enumerate}
\item $\delta (L)$ is a decreasing function of $r$, varying between 2 and 1. $\delta(L)=2$ if and only if $L$ is (numerically) subcanonical and $\delta (L)=1$ if and only if $r(L)=\infty$.
\item $\delta (L)$ is a decreasing function of $L$, i.e., if $L_1\preceq L_2$ then $\delta (L_1) \geq \delta (L_2)$.
\item $\delta (L)$ increases by hyperplane section, i.e., if $M$ is a smooth section of a nef line bundle then $\delta (L_{|M}) \geq \delta (L)$.
\item If $\mu: {\widetilde X} \longrightarrow X$ is an \'etale Galois covering and ${\widetilde L}={\mu}^*L$, then $\delta ({\widetilde L})=\delta (L)$.
\item $\delta (L)$ increases by blow-up, i.e., if $\sigma: X' \longrightarrow X$ is any blow-up, then $\delta (\sigma^*(L))\geq \delta (L)$.
\item If $X$ is a variety of general type, then $K_X$ is big and hence for all $L$ there exists an $r$ such that $rK_X-L$ is effective. Then we always have $\delta (L) > 1$. The case $X=A$ an abelian variety is just the opposite: $\delta({\mathcal O}_A)=2$ and for all nef $L \neq {\mathcal O}_A$ we have $\delta (L) =1$.
\end{enumerate}
\end{remark}

\section{Some properties of continuous linear systems}

Let $a: X \longrightarrow A$ be a nontrivial map to an abelian variety and let $L\in {\rm Pic(X)}$. We are going to study the geometry of the continuous linear systems $|L\otimes \alpha|$ for $\alpha$ general. A good presentation and an analysis of the generic base loci of the {\it main continuous system} associated to $L$ is developed in \cite{MLPP1} and \cite{MLPP2}. Here we consider two related problems: the continuous resolution of base points of a continuous linear system  (Theorem 3.2), and the behavior of a general $|L\otimes \alpha|$ up to an \'etale covering (Theorem 3.5).

It is well known that there exists a nonempty open set

$$U\subseteq U_{a,L}=\{ \alpha \in {\widehat A} \,|\,h^0(X,L\otimes \alpha)=h^0_a(L) \}$$

 \noindent such that for $\alpha \in U$, if we consider the decomposition $L\otimes \alpha=W_{\alpha}+N_{\alpha}$ into its fixed and moving part respectively, then the divisors ${W_{\alpha}}$ belong to the same algebraic class, and the same occurs with the divisors $N_{\alpha}$. This is basically the construction given in \cite{BLNP}, section 5.1, for the canonical line bundle and ${\rm Pic}^0(X)$, but it holds in general. Roughly speaking, consider in $X\times {\widehat A}$ the closed set ${\cal D}=\{(p,\alpha)\,|\, p\in {\rm Bs}(|L\otimes \alpha|)\,\}$. When $|L\otimes \alpha|$ has a divisorial base component, ${\cal D}$ has a codimension 1 component which is dominant with respect to the second projection. On an open set $U$ of ${\rm Pic^0(X)}$ its fibres are algebraically equivalent.

We define now the {\it continuous moving part} and the {\it continuous fixed part} of $L$ as follows. Consider first the evaluation map

$$ev_U:=\oplus ev_{\alpha}: \bigoplus_{\alpha \in U}H^0(X,L\otimes \alpha) \otimes {\alpha}^{-1} \longrightarrow L.$$

We have that ${\rm Im}ev_U={\cal I}_U\otimes L$, where ${\cal I}_U$ is an ideal sheaf. In the proof of Lemma 3.2 we will see that this sheaf does not depend on the chosen open set $U$ verifying the conditions above. Consider its decomposition

$${\cal I}_U={\cal O}_X(-W)\otimes {\cal I}_B$$

\noindent with ${\rm codim}_{X}B\geq 2$.

\begin{definition}\label{continuousmovingpart}
\begin{itemize}
\item [(i)] $W$ is the {\it continuous fixed part} of $L$.
\item [(ii)] $M=L-W$ is the {\it continuous moving part} of $L$.
\end{itemize}
\end{definition}

Set ${\rm Im}ev_{\alpha}={\cal I}_{\alpha}\otimes L\otimes \alpha$ where ${\cal I}_{\alpha}={\cal O}(-W_{\alpha})\otimes {\cal I}_{B_{\alpha}}$ with ${\rm codim}_X(B_{\alpha})\geq 2$. Observe that, by construction we have

$${\cal I}_{\alpha}\otimes L\otimes \alpha \subseteq M\otimes \alpha \subseteq L\otimes \alpha.$$

\noindent Hence we have that

$$h^0_a(L)=h^0_a(M).$$

Following Pareschi and Popa (\cite{Pareschipopa2}), recall that a line bundle $L$ on $X$ is  {\it continuously globally generated with respect to the map $a$} if the continuous evaluations maps $ev_V$ defined above are surjective for {\it all} non-empty open sets $V \subseteq {\widehat A}$. First of all we have

\bigskip

\begin{lemma} With the previous notation, the sheaf ${\cal F}={\rm Im}(ev_U)$ is continuously globally generated with respect to $a$.
\end{lemma}

\begin{proof} This is the content of Remark 4.2 in \cite{BLNP}. To sketch a proof observe that given a point $p\in X$, if there exists a nonempty open subset $V\subseteq U$ such that $p$ is a base point of the linear systems $|L\otimes \alpha|$ for all $\alpha \in V$, then $p$ is also a base point of these linear systems for all $\alpha \in U$. Indeed, all the sections in $H^0(X,L\otimes \alpha)$, $\alpha \in U$, are limits of those with $\alpha \in V$, since on $U$ the dimensions $h^0(X,L\otimes \alpha)$ are constant.

Hence, for any open set $V \subseteq {\widehat A}$ we have

$${\rm Im}(ev_V)\supseteq {\rm Im}(ev_{V\cap U})={\rm Im}(ev_U)={\cal F}.$$

Finally, observe that by construction we have for any $\alpha \in U$

$${\cal I}_{\alpha}\otimes (L\otimes \alpha)\subseteq {\cal F}\otimes \alpha \subseteq L\otimes \alpha$$

\noindent and hence $H^0(X,{\cal F}\otimes \alpha)=H^0(X,L\otimes \alpha)$ and so we have ${\rm Im}(ev_{U,L})={\rm Im}(ev_{U,{\cal F}})$.
\end{proof}

\smallskip

Now we are going to see that the continuous fixed and moving parts of $L$ behave as the linear ones up to a suitable \'etale covering.

\bigskip

\begin{theorem}\label{continuousresolution} Let $X$ be a smooth projective irregular variety and let $a: X \longrightarrow A$ be a nontrivial map to an abelian variety, such that $a^*: {\widehat A} \longrightarrow {\rm Pic}^0X$ is injective. Then, up to a blow-up $\sigma$ and an \'etale Galois covering (more concretely, a base change by a multiplication map) $\mu$
\begin{equation}\label{continuousbasepoints}
\xymatrix {{\widetilde X}\ar[r]^{\mu} \ar[d]^{\widetilde a}& X' \ar[r]^{\sigma}\ar[d]^{a'} & X \ar[ld]^{a}\\A \ar[r]^{\mu_d} &A}
\end{equation}
\noindent we have that for all $\alpha \in {\widehat A}$

$$\lambda^*(L\otimes \alpha)={\widetilde W}+{\widetilde N_{\alpha}}$$

\noindent is the decomposition in the fixed and moving divisor, $\lambda=\mu \circ \sigma$, the linear system $|{\widetilde N_{\alpha}}|$ is base point free and the divisor ${\widetilde W}$ does not depend on $\alpha$.
\end{theorem}

\begin{proof}Consider a blow-up $\sigma: X' \longrightarrow X$ such that $\sigma ^*({\cal O}_X(-W)\otimes{\cal I}_B)={\cal O}_{X'}(-W')$. If we set $L'=\sigma^*L$, then we have that ${\sigma}^*{\cal F}=L'(-W')$ is continuously globally generated with respect to $a'=a \circ \sigma$. Then we can apply a result of Debarre (cf. \cite{De} Proposition 3.1): there exists an \'etale Galois covering $\mu: {\widetilde X} \longrightarrow X'$ which we can assume is induced by a multiplication map on $A$, such that for any $\alpha \in {\widehat A}$, the line bundles ${\widetilde L}(-\mu^*W')\otimes {\alpha}={\mu}^*L'(-W')\otimes \alpha$ are {\it globally generated}. Hence, if we set ${\widetilde W}=\mu^*W'$ we have the statement. The proof of Debarre is for ${\rm alb}_X$ but it also works for any map $a$ such that $a^*$ is injective (hence ${\widetilde X}$ is connected).
\end{proof}

\begin{remark} Observe that in the previous construction we have the following properties
\begin{itemize}
\item For all $\alpha \in \widehat A$ the line bundles ${\widetilde N}_{\alpha}$ on ${\widetilde X}$ are algebraically equivalent.
\item $h^0_{\widetilde a}({\widetilde N_{\alpha}})=h^0_{\widetilde a}({\widetilde L})=({\rm deg}\mu)h^0_{a'}(L')=({\rm deg}\mu)h^0_a(M)=({\rm deg}\mu)h^0_a(L)$.
\item For all $l\geq 1$ we have $h^0({\widetilde X}, {\widetilde N_{\alpha}}^{\otimes l})=({\rm deg}\mu)h^0(X,(M\otimes \alpha)^{\otimes l})$.
\item Since ${\widetilde N_{\alpha}}$ are base point free, if $L$ is nef then $({\rm deg}\mu)L^n=({\widetilde L})^n\geq ({\widetilde N_{\alpha}})^n$.
\end{itemize}
\end{remark}

To finish the section we are going to see that given a nef line bundle with continuous sections, the image of the map induced by its linear sections, up to \'etale base change, factors the map $a$. Hence the dimension of the image of $X$ through the linear system $|L|$ is bounded by the $a$-dimension of $X$. In particular, it is maximal when $X$ is of maximal $a$-dimension. This will be a crucial point in the proof of the Main Theorem.

In order to do this, consider the following notation. Given a line bundle $L$ with nontrivial sections consider the morphism $\psi_L:X' \longrightarrow \mathbb{P}^m$ it induces on a suitable blow-up $X'$ of $X$. We denote by

$$\phi_L: X' \longrightarrow Z_L$$

\noindent the algebraic fibre space induced by $\psi_L$.

\begin{theorem} \label{etale} Let $X$ be a smooth n-dimensional variety and $a:X \longrightarrow A$ a nontrivial map to an abelian variety such that $a^*: {\widehat A} \longrightarrow {\rm Pic}^0X$ is injective. Let $k={\rm dim} \, a(X)$. Let $L\in{\rm Pic}X$, such that $h^0_a(L)\neq 0$. Let $M$ be its continuous moving part. Consider the previous notation of this section and the map $\lambda$ given by Theorem \ref{continuousresolution}. Then
\begin{itemize}
\item [(i)] There exists a factorization $\widetilde a=\phi _{\widetilde L}\circ a'$, where $a':Z_{\widetilde L}\longrightarrow A$. In particular ${\rm dim} \, \psi_{\widetilde L}({\widetilde X})\geq k$. Moreover, for all $\alpha \in {\widehat A}$ we have that $N_{\alpha}\in \phi_{\widetilde L}^*{\rm Pic}Z_{\widetilde L}$.
\item [(ii)] The linear system $|{\widetilde L}|$ induces a generically finite map provided one of the following conditions hold
  \begin{itemize}
  \item $|L|$ induces a generically finite map.
  \item $X$ is of maximal $a$-dimension.
  \item $M$ is $a$-big.
  \end{itemize}
\end{itemize}
\end{theorem}

\begin{proof}

\noindent (i) Let $T$ be a general (connected) fibre of $\phi _{\widetilde L}$, and let $R\in {\rm Pic}(Z_{\widetilde L})$ such that ${\widetilde N}_0=\phi_{\widetilde L}^*(R)$. By Remark 3.4, if $h^0_a(L)\neq 0$ then $h^0_{\widetilde a}({\widetilde N}_0)\neq 0$ and hence for all $\alpha \in {\widehat A}$ we have

$$h^0({\widetilde X},{\widetilde N}_0\otimes {\widetilde a}^*{\alpha})\neq 0.$$

By projection formula we have that $h^0(Z_{\widetilde L},R \otimes (\phi_{\widetilde L})_*({\widetilde a}^*\alpha))\neq 0$ for all $\alpha \in {\widehat A}$. The sheaf $(\phi_{\widetilde L})_*({\widetilde a}^*\alpha)$ is torsion free of rank $h^0(T,({\widetilde a}^*\alpha)_{|T})$ and so it must be non zero for all $\alpha$. Since $({\widetilde a}^*\alpha)\in {\rm Pic}^0(T)$ this can only happen if $({\widetilde a}^*\alpha)_{|T}={\mathcal O}_T$. Hence the natural composition map

$${\widehat A} \longrightarrow {\rm Pic}^0 ({\widetilde X}) \longrightarrow {\rm Pic}^0 (T)$$

\noindent is zero. Dualizing we obtain that for general $T$ the map ${\widetilde a}$ contracts $T$ to a point.

The rest of the statement follows immediately from this factorization.

(ii) If $|L|$ induces a generically finite map clearly so does $|{\widetilde L}|$. By (i) the same holds if $X$ is of maximal $a$-dimension ($k=n$). For the rest, observe that the fibres $G'$ of the map $a \circ \lambda$ are just disconnected copies of $G$. On the other hand, by construction ${\widetilde N}_{0|G'}$ is big if and only if $M_{|G'}$ is, hence we can assume that ${\widetilde N}_0$ is ${\widetilde a}$-big. But if $r<n$ this is not possible since the fibres $G'$ are covered by those of $\phi_{\widetilde L}$.

\end{proof}

\begin{remark} Clearly $a$-bigness of $M$ follows from bigness of $M$ itself, but in general it can be difficult to check. Here we have three sufficient conditions.
\begin{itemize}
\item If $L$ is $a$-big (for example, if $L$ itself is big) and kod$(G,W_{|G})\leq 0$, then $M$ $a$-is big.
\item If $L$ is continuously globally generated in codimension 2 (i.e. $W=\emptyset$), then bigness of $M$ is equivalent to bigness of $L$, i.e., $L^n>0$.
\item Continuous global generation of $L$ outside of the ramification locus of the Albanese map of $X$, is implied by $M$-regularity of $L$ (see \cite{Pareschipopa2}).
\end{itemize}
\end{remark}

\begin{remark} In general $a$-bigness of $L$ does not imply $a$-bigness of its continuous moving part $M$. Take for example $X=S\times Y$ where $S$ is a general type surface with $p_g=1$, $q=0$ and $Y$ is a general type and albanese general type variety of dimension $n-2$, with base point free paracanonical system $N$. Then clearly $L=\omega_X$ is big, its continuous fixed part is $Z=\pi_1^*(D)$ ($D$ the only canonical section of $S$) and $M=\pi_2^*N$, the Albanese map of $X$ is just the Albanese map of $Y$ and $M_{|G}$ is not big.
\end{remark}

\section{The generalized Clifford-Severi Inequality}

In the previous sections we have introduced all the ingredients needed to state and prove our results. All of them are particular instances or corollaries of the Main Theorem in the introduction which we reproduce here.

\begin{theorem}\label{CS}({\bf Generalized Clifford-Severi Inequality})

Let $X$ be a smooth, projective variety of dimension $n$, over an algebraically closed field of characteristic 0. Let $a: X \longrightarrow A$ be a nontrivial map to an Abelian variety and let $L \in {\rm Pic}(X)$ be a nef line bundle.
\begin{itemize}

\item [(i)] If $X$ is of maximal $a$-dimension then $$L^n \geq  \delta (L)\, n! \,h^0_a(L).$$
\noindent In particular, if $L \preceq K_X$, then $L^n \geq 2n! \, h^0_a(L).$
\item [(ii)] Assume $n> {\rm dim} \, a(X)=k\geq 1$ and let $M$ be the continuous moving part of $L$. If $M$ is $a$-big then
$$L^n \geq  \delta (L)\, k! \,h^0_a(L).$$
\item [(iii)] Assume that $n>{\rm dim} \, a(X)=k\geq 1$ and that $L$ is $a$-big. Then  $$L^n \geq  k! \,h^0_a(L).$$
\end{itemize}
\end{theorem}

The proof of this theorem relies on a suitable use of Xiao's method on \'etale Galois coverings of $X$ and it is postponed to the next Section.

\begin{remark}\label{otraversion} Since $\delta$ is a decreasing function, we do not need to know the exact value of $r(L)$ to obtain an inequality. Hence, part (i) of the theorem can be rephrased as (analogously for (ii)): when $X$ is of maximal $a$-dimension
\begin{itemize}
\item [(i1)] For any nef $L$ we have $L^n \geq n! \, h^0_a(L)$.
\item [(i2)] If $L \preceq rK_X$ ($r \geq 1$), then $L^n \geq \frac{2r}{2r-1}\,n!\,h^0_a(L)$.
\end{itemize}

\end{remark}

\begin{remark}\begin{itemize}
\item [(i)] The bound given in \ref{CS} (i) is sharp for the lowest value of $\delta(L)$: take $X=A$ an abelian variety and $L$ an ample line bundle on $A$. Also, in general it is asymptotically sharp for general type varieties $X$ and sufficiently ample line bundles $L=mH$, as shown by asymptotic Riemann-Roch theorem when $m \rightarrow \infty$.
\item [(ii)] For varieties of non maximal Albanese dimension Theorem \ref{CS} cannot hold without extra hypothesis as those in (ii) and (iii). Indeed, take $L$ to be the pullback of a line bundle on the Albanese image of $X$. Then $L^n=0$ and in general $h^0_{alb_X}(L)\neq 0$.
\end{itemize}
\end{remark}

\begin{remark} As shown in Remark \ref{ejemplos}, the continuous rank $h^0_a(L)$ is just $\chi (X,L)$ or even $h^0(X,L)$ under conditions of positivity of $L$ and of maximal $a$-dimension of $X$,  which produce (generic) vanishing of higher cohomology.
\end{remark}

\bigskip

We proceed now to prove the corollaries A,B and C stated in the Introduction.

\begin{proof}{\it Corollary A.} Let $Y=\mathbb{P}_X({\cal F})$, $\pi: Y \longrightarrow X$ the natural projection and $L={\cal O}_Y(1)$ the tautological line bundle.  Consider the induced map $a_Y:=a\circ \pi: Y \longrightarrow A$; we have ${\rm dim} \, a_Y(Y)=k$. Observe that  $s({\cal F})={\cal O}_Y(1)^{n+l-1}$ where $l={\rm rank}{\cal F}$.

Let $G'$ be the fibre of the algebraic fibre space induced by the Albanese map of $Y$.
If $k=n$ then $G'=\mathbb{P}^{l-1}$, $L_{|G'}={\cal O}_{\mathbb{P}^{l-1}}(1)$ and hypothesis of theorem \ref{CS} (ii) holds.
If $k<n$ then $G'$ is a projective bundle on $G$. If ${\cal F}_{|G}$ is big, then the hypothesis of (iii) holds.  In both cases we use $\delta(L) \geq 1$.
\end{proof}

\begin{proof}{\it Corollary B.}  Since $X$ is normal we can consider the canonical Weil divisor $K_X$ and set $\omega_X={\cal O}_X(K_X)$ for the associated divisorial sheaf. Consider a desingularization $\sigma: X' \longrightarrow X$; since $X$ is minimal it has terminal singularities and we have that ${\rm Alb}X={\rm Alb}X'=:A$ and ${\rm alb}_{X'}={\rm alb}_{X}\circ \sigma$ (\cite{BS}, Ch.2.4). Terminal singularities also give $H^0(X',\omega_{X'})\cong H^0(X,\sigma _* \omega_{X'})=H^0(X,\omega_X)$ and so $h^0_{{\rm alb}_X}(\omega_X)=h^0_{{\rm alb}_{X'}}(\omega_{X'})$ and it coincides with $\chi(\omega_{X'})$ when $X$ is of maximal Albanese dimension.

To avoid working with the singularities of $X$, we are going to show now how to reduce the computation to a suitable {\it line bundle} $M$ on $X'$ such that $K_X^n \geq M^n$ and $h^0_{alb_{X'}}(M)=h^0_{alb_{X'}}(\omega_{X'})$.

Since the inequalities are invariant through \'etale covers, by Theorem \ref{continuousresolution} we can consider an extra blow-up and a base change through a multiplication map and assume that $K_{X'}=W+M$, where $W$ and $M$ are the continuous fixed divisor and the continuous moving divisor of $K_{X'}$ respectively and the linear system $|M|$ is base point free. Moreover, as seen after Definition 3.1, we have that $h^0_{alb_{X'}}(M)=h^0_{alb_{X'}}(\omega_{X'})$. Since $|M|$ is base point free we have an induced morphism $\lambda': X' \longrightarrow \mathbb{P}':=\mathbb{P}(H^0(X',\omega _{X'})^*)$ such that $M={\lambda'}^*T_{\mathbb{P}'}$ where  $T_{\mathbb{P}'}={\cal O}_{\mathbb{P}'}(1)$.

% Fix $\alpha_0 \in {\widehat A}$ such that $h^0(X', \omega_{X'} \otimes \alpha_0)=h^0_{{\rm alb}_{X'}(\omega_{X'})$.

Since the variety $X$ is normal and $\mathbb{Q}$-factorial, we can apply a $\mathbb{Q}$-{\it resolution of base loci} for Weil divisors. We claim that there exists a desingularization (which we still call $X'$) $\sigma: X' \longrightarrow X$, an effective Weil divisor $E_0$ on $X$ and an effective $\sigma$-exceptional $\mathbb{Q}$-divisor $E$ on $X'$ such that

$$\sigma^*(K_X-E_0)-E \sim_{\mathbb{Q}} (\lambda \circ \sigma)^* T_{\mathbb{P}}$$

\noindent where $\lambda: X \dashrightarrow \mathbb{P}(H^0(X,\omega_X)^*)$ is the natural map and $\sim_{\mathbb{Q}}$ means $\mathbb{Q}$-linear equivalence. The natural identification $\mathbb{P}(H^0(X',\omega _{X'})^*) =\mathbb{P}(H^0(X,\omega_X)^*)$ induces $\lambda'=\lambda \circ \sigma$ and so $(\lambda \circ \sigma)^* T_{\mathbb{P}}=(\lambda')^*T_{\mathbb{P}'}=M$.

Indeed, the claim is an absolute version of Lemma 1.1 in \cite{Ohno}. We sketch here the proof and refer there for details. Consider $E_0$ a Weil divisor on $X$ such that  $ev: H^0(X,K_{X})\otimes {\cal O}{_X} \longrightarrow {\cal O}_{X}(K_{X}-E_0)$ is surjective in codimension 1. Take a positive integer $m$ such that $m(K_{X}-E_0)$ is Cartier.  Up to a suitable resolution of singularities and base loci, we obtain a desingularization $\sigma: X' \longrightarrow X$ and an effective divisor ${\overline E}$ on $X'$ such that the natural morphism $\sigma^*(Sym^mH^0(X,K_{X})\otimes {\cal O}_{X})\longrightarrow \sigma^*(m(K_X-E_0))\otimes {\cal O}_{X'}(-{\overline E})$ is surjective. Hence we obtain

$$\sigma^*(m(K_X-E_0))-{\overline E}=\lambda_m^*T_{\mathbb{P}((Sym^mH^0(K_X))^*)}=m(\lambda \circ \sigma)^*T_{\mathbb{P}(H^0(X,K_X)^*)}=mM$$

\noindent where $\lambda_m$ is the natural morphism  $X' \longrightarrow {\mathbb{P}((Sym^mH^0(K_X))^*)}$. The claim then follows just dividing by $m$ and defining $E:=\frac{1}{m}{\overline E}$.

In order to conclude the proof of the corollary, just observe that $M$ and $K_X$ are nef and that $E_0$ and $E$ are effective. Hence

$$K_X^n=(\sigma^* K_X)^n\geq M^n.$$

Then the results follow from Main Theorem applied to $(X',M)$, using that $\delta(M)=2$ since it is subcanonical.

\end{proof}

\begin{remark} Irregular, general type varieties of Albanese dimension $k$, with ${\rm vol}(X)< 2k!\chi(\omega_X)$ seem to have strong restrictions, if they exist. At least when the minimal model $X$ is Gorenstein, the albanese map of $X$ must factor through a fibration with general fibre $G$ of dimension $l\geq 1$ such that $K^l_G=1$
 %(hence regular or with $\chi(\omega_G)\leq 1$)
as follows as a corollary of proof of Theorem \ref{CS} (see Remark 5.8). In particular, for minimal, Gorenstein $X$ of Albanese dimension $(n-1)$, the inequality $K^n_X\geq 2 (n-1)! \,\chi(\omega_X)$ also holds.

\end{remark}

\medskip

\begin{proof}{\it Corollary C} (i) If $b=0$ just add up the Severi inequalities for $K_X$ and for $K_F$

$$K_f^n=K_X^n+2nK_F^{n-1}\geq 2n!\,\chi (\omega_X)+4n!\,\chi(\omega
_F)=2n!\,(\chi(\omega_X)-\chi(\omega_F)\chi(\omega_B))=2n!\,\chi_f$$

(ii) If $b\geq 1$ then $\omega_f$ is subcanonical and so $\delta(\omega
_f)=2$. Then apply Proposition \ref{h0slope} and Theorem \ref{CS}.

\end{proof}

%The proof of \ref{CS} is based in the following proposition, which plays the role of the slope inequality in Pardini's proof of the classical Severi Inequality. We % postpose the proof of this technical result to the next section.

%\begin{proposition}\label{Clifford} Let $X$ be a smooth projective variety $X$, over a field of characteristic 0 and of dimension $n\geq 1$. Let $a: X %\longrightarrow A$ be a map on an abelian variety such that $X$ is of maximal a-dimension. Let $H$ be a very ample sheaf on $A$ and denote $M=a^*H$. Then, for %any  nef line bundle $L$ on $X$, we have

%$$\sum_{i=0}^n \frac{n!}{i!}L^iM^{n-i} \geq \delta (L) n! h^0_a(L)$$

%$$(1+\frac{\sigma^n _0}{L^{n-1}M})L^n + \sum_{i=1}^n\sigma^n _i L^{n-i}M^i \geq \delta (L) n! h^0(X,L)$$
%\end{proposition}

 Main Theorem may be an useful tool for classification of irregular varieties. For example, we present here a stronger version of Severi inequality for surfaces with numerically decomposable canonical bundle, which is useful for the classification of irregular surfaces of small invariants (see for example \cite{CMLP} for recent progress).

\begin{proposition} Let $S$ be a smooth surface of maximal Albanese dimension such that $K_S\equiv L_1 + L_2$ with $L_i$ nef line bundles. Then

$$K^2_S \geq 4{\chi}(\omega _S) + 4 h^1_a(L_1)$$
\end{proposition}
\begin{proof} Let $\rho \in {\rm Pic}^{\tau}(S)$ such that $L_1 +(L_2+\rho)=K_S$ and redefine $L_2=L_2+\rho$. Observe that $h^i_a(L_1)=h^{2-i}_a(L_2)$ and both are (numerically) sub-canonical. Applying Theorem (\ref{CS}) to both sheaves and adding up, we obtain

$$L_1^2 +L_2^2 \geq 4 (h^0_a(L_1)+h^0_a(L_2))=4\chi(S,L_1)+4h^1_a(L_1)=2L_1(-L_2)+4\chi(\omega_S)+4h^1_a(L_1)$$

\noindent and hence

$$K^2_S=L_1^2+2L_1L_2+L_2^2 \geq 4 \chi(\omega_S) +4h^1_a(L_1)$$

\end{proof}

\section{Proof of Main Theorem}

As pointed out in the introduction, the proof relies on three basic tools: Xiao's method, the behavior of linear systems on suitable \'etale coverings (studied in Section 3) and finally Pardini's covering trick.

\subsection{Xiao's method}\label{subsectionxiao}

We will use a simplified version of Xiao's method for fibrations in the case where the base curve is $\mathbb{P}^1$, and specially adapted to an ulterior process of \'etale covers. We remind briefly the method in this simplified version and refer to \cite{survey}, \cite{Konno}, \cite{Ohno} and \cite{Xiao} for details. The construction holds in general for $\mathbb{Q}$-Cartier Weil divisors but for simplicity we state it for Cartier divisors, which is the case we will use.

Let $X$ be a normal projective variety of dimension $n$ and let $D$ be a nef Cartier divisor. Let $L={\cal O}_X(D)$ be its associated line bundle. Assume we have a fibration $f: X \longrightarrow \mathbb{P}^1$ with $F$ a general fibre of $f$ and let ${\mathcal E}=f_*L$. It is a vector bundle since it is torsion free on a smooth curve. Consider its decomposition $${\mathcal E}=f_*L=\bigoplus_{i=1}^{l} {\mathcal O}_{\mathbb{P}^1}(a_i)$$

\noindent with $a_1 \geq a_2 \geq ...\geq a_m \geq 0 > a_{m+1} \geq ... \geq a_l$, $\l \geq m\geq 0$, and $l=h^0(F,L_{|F})$. Observe that we have $h^0(X,L)=a_1 +...+a_m+m$ and so

\begin{equation}\label{a1am}
a_1+...+a_m \geq h^0(X,L)-h^0(F,L_{|F}).
\end{equation}

%\noindent {\underline {\it Remark}} This is a particular case of the following result: given a nef vector bundle ${\cal E}$ on a curve, the inequality ${\rm %deg}{\cal E} \geq h^0(C,{\cal E})-{\rm rank} {\cal E}$ holds.

%\smallskip

For $i=1,...,m$, define ${\cal E}_i={\cal O}_{\mathbb{P}^1}(a_1)\oplus ...\oplus {\cal O}_{\mathbb{P}^1}(a_i)$. When the $a_i$'s are different, these are the pieces of the Harder-Narashiman filtration of ${\cal E}_m$, of associated slopes $\mu _i = a_i$.

For each $i=1,...,m$ such that $a_i > a_{i+1}$ the composite of the natural sheaf homomorphisms
$$
f^*{\mathcal E}_i(-a_i)\rightarrow f^*(f_*L)(-a_i)\rightarrow L(-a
_iF)
$$
\noindent surjects onto a sheaf of ideals of type ${\mathcal I}_{Z_i}\otimes L(-a_iF)$. Following \cite{Ohno} Lemma 1.1 and Remark therein, up to a suitable desingularization $\epsilon: {\widehat X} \longrightarrow X$, if we set ${\widehat L}=\epsilon ^* L$ and ${\widehat F}=\epsilon ^* F$, we have a decomposition

$${\widehat L}=N_i+{\widehat Z}_i+a_i{\widehat F}$$

\noindent where:
\begin{itemize}
\item $N_i$ is a nef Cartier divisor on ${\widehat X}$ inducing a base point free linear system.
\item ${\widehat Z}_i$ is an effective and fixed Cartier divisor (the base divisor of the induced linear system on ${\widehat X}$).
\end{itemize}

%we can assume that $Z_i$ are divisors and so on ${\widehat X}$ we have ${\mathcal I}_{{\widehat Z}_i}\otimes {\widehat L}(-a_i)={\widehat L}(-a_i{\widehat %F}-{\widehat Z}_i)=:N_i$ is a line bundle which induces a base point free linear system on ${\widehat X}$.

If $a_i=a_{i+1}$ we define $N_i=N_{i+1}$, ${\widehat Z}_i={\widehat Z}_{i+1}$. Observe that $N_1={\mathcal O}_{\widehat X}$ if and only if $a_1>a
_2$. We redefine $a_{m+1}=0$ and extend coherently the definition to $N_{m+1}$ and ${\widehat Z}_{m+1}$.

%that all the  $N_i$ are effective and nef and that ${\widehat Z}_i$ are effective and fixed.

Moreover, we have that

$$N_1 \leq N_2 \leq ...\leq N_m \leq N_{m+1} \leq L$$
$${\widehat Z}_1 \geq {\widehat Z}_2 \geq ... \geq {\widehat Z}_m \geq {\widehat Z}_{m+1} \geq 0$$
$$a_1 \geq a_2 \geq ... \geq a_m \geq a_{m+1}=0$$

In fact, by construction we have that

$$N_i+({\widehat Z}_i-{\widehat Z}_{m+1})=N_{m+1}(-a_i{\widehat F})$$

\noindent is the decomposition  of $N_{m+1}(-a_i{\widehat F})$ in its moving and fixed part, respectively.

%\noindent and for all $i=1,...,m$

%$$L=N_i+a_iF+Z_i$$

Under these assumptions we can apply Xiao's Lemma (see \cite{Konno}) and \cite{Ohno} Lemma 1.2.
We define the linear systems $P_i:=N_i|_{\widehat F}$
which are free from base points and induce maps $\phi _i: \widehat F \longrightarrow {\mathbb P}^{r_i-1}$. Observe that for $i=1,...,m$ we have $r_i \geq i$.

%By construction we have $P_m \geq P_{m-1} \geq ... \geq P_2 \geq P_1= {{\cal O}_{\widehat F}}$.

%$\epsilon: {\widehat X} \longrightarrow X$, the above map becomes a morphism for every
%$i$.
%induces a rational map $X \rightarrow {\mathbb
%P}_{\mathbb{P}^1}({\mathcal E}_i)$. Up to a suitable sequence of blowing-ups
%$\epsilon: {\widehat X} \longrightarrow X$, the above map becomes a morphism for every
%$i$.

%Let $N_i$ be the moving part of the pull-back of the
%tautological line bundle  on ${\mathbb P}_{\mathbb{P}^1}({\mathcal E}_i)$.

%Redefine $a_{m+1}=0$ and $N_{m+1}=N_m$. Then, we can apply the generalized Xiao's inequality as given by Konno in (\cite{Konno}).

Define now

$$I_s=\{k=1,...,m \, | \, {\rm dim} \, \phi _k({\widehat F})=s \, \}$$

\noindent and we obtain a partition of the set $\{ \, 1,...,m \, \}$.

Let $r$ be the maximum index such that $I_{r-1}\neq \emptyset$ and define decreasingly, for $s=1,...,r-1$

$$b_s=\{
           \begin{array}{cc}
             {\rm min} I_s & {\rm if} \,\, I_s \neq \emptyset \\
             b_{s+1} & {\rm otherwise} \\
           \end{array}$$

 Then we have that, for any $A_1,...,A_{n-r}$ nef Cartier divisors the following inequality holds:

 $$A_1...A_{n-r}\left[N^r_{m+1}-(\sum _{s=r-1}^{1} (\prod_{k>s}P_{b_k})\sum_{i \in I_s}(\sum_{l=0}^{s}P_i^{s-l}P_{i+1}^{l})(a_i-a_{i+1}))\right]\geq 0.$$

 \noindent In particular, taking $A_1=...=A_{n-r}={\widehat L}$ we obtain

$$L^n=({\widehat L})^n \geq {\widehat L}^{n-r}N_{m+1}^r\geq {\widehat L}^{n-r}\left[\sum _{s=r-1}^{1} (\prod_{k>s}P_{b_k})\sum_{i \in I_s}(\sum_{l=0}^{s}P_i^{s-l}P_{i+1}^{l})(a_i-a_{i+1})\right]$$

Since $P_{i+1} \geq P_i$ and they are nef, we have that $$\sum_{l=0}^{s}(P_i^{s-l}P_{i+1}^{l})\geq (s+1)P_i^s$$

\noindent and so

\begin{equation}\label{xiaobasica}
L^n \geq {\widehat L}^{n-r}\left[\sum _{s=r-1}^{1} (s+1)(\prod_{k>s}P
_{b_k})\sum_{i \in I_s}P_i^s(a_i-a_{i+1})\right]
\end{equation}

For later reference we need to consider the following special case. Assume that all the induced maps $\phi_i$ have image of dimension $n-1$, i.e., they are generically finite. Then

%\begin{equation}\label{xiaotruncada}
%L^n\geq rL^{n-r}[P_m^{r-1}(a_m-a_{m+1})+P_{m-1}^{r-1}(a_{m-1}-a
%_m)+...+P_1^{r-1}(a_1-a_2))]
%\end{equation}

%Under these assumptions, Stein factorization of the maps $\phi_i$ gives an algebraic fiber space $\rho: {\widehat F} \longrightarrow Y$ with ${\rm %dim}Y=r-1$ and finite maps $\psi_i: Y \longrightarrow \phi_i({\widehat F})$ induced by linear systems $|R_i|$ on $Y$. Let $G$ be the general fibre of 5$\rho$. Then $P_i^{r-1}=(R_i^{r-1}G$ and so inequality \ref{xiaotruncada} reads

%\begin{equation}\label{xiaotruncadabis}
%L^n\geq r(L^{n-r}G)[R_m^{r-1}(a_m-a_{m+1})+R_{m-1}^{r-1}(a_{m-1}-a_m)+...+R_1^{r-1}(a_1-a_2))]
%\end{equation}

%When $r=n$, i.e., when all the induced maps on ${\widehat F}$ are generically finite we have

\begin{equation}\label{xiaotruncadagenericamentefinito}
L^n\geq n[P_m^{n-1}(a_m-a_{m+1})+P_{m-1}^{n-1}(a_{m-1}-a_m)+...+P_1^{n-1}(a_1-a_2))]
\end{equation}

%\begin{equation}\label{xiaotruncada}
%L^n\geq n[P_m^{n-1}(a_m-a_{m+1})+P_{m-1}^{n-1}(a_{m-1}-a_m)+...+P_{b_{n-1}}^{n-1}(a_{b_{n-1}}-a_{b_{n-1}+1})]+ \Delta
%\end{equation}
%\noindent where

%\begin{equation}\label{Delta}
%\Delta=\sum _{s=n-2}^{1} (s+1)(\prod_{k>s}P_{b_k})\sum_{i \in I_s}P_i^s(a_i-a_{i+1})
%\end{equation}

%\noindent is the contribution to the formula of the linear systems on ${\widehat F}$ which decrease the dimension.

\begin{remark}\label{casoespecialxiao} The same construction can be applied to any selection of indexes $I \subseteq \{1,...,m\}$.
\end{remark}

We are going to see now how the method behaves under a suitable \'etale Galois covering of $X$. It turns out that a {\it continuous Xiao's method} holds.

\begin{proposition} Let $X$ be a normal projective variety and $L$ a line budle. Let $a: X \longrightarrow A$ be a nontrivial map to an abelian variety and $f: X \longrightarrow \mathbb{P}^1$  a fibration. Then, for a very general element $\beta \in {\widehat A}$, $L':=L\otimes \beta$ verifies that the vector bundles $\cE_{\alpha}=f_*(L'\otimes \alpha)$ are all equal for $\alpha \in {\widehat A}_{{\rm tors}}$.
\end{proposition}

\begin{proof} The continuous family of vector bundles $\{f_*(L\otimes \beta)\}$ on $\mathbb{P}^1$ for $\beta \in {\widehat A}$ must be constant on a nonempty open set $U$. Let $D={\widehat A}\setminus U$. Then

    $$\bigcup_{\alpha \in {\widehat A}_{{\rm tors}}}(\alpha + D)\neq {\widehat A}$$

    \noindent being a countable union of proper closed sets. An element $\beta$ in its complementary set verifies the statement.

\end{proof}

Consider  a fibration $f: X \longrightarrow \mathbb{P}^1$ and a nontrivial map $a:X \longrightarrow A$ such that $a^*$ is injective. Let $L$ be a nef line bundle such that $h^0_a(L) \neq 0$ and which verifies the conclusion of Proposition 5.5, i.e., for all $\alpha \in {\widehat A}_{{\rm tors}}$, the sheaves $f_*(L\otimes\alpha)$ are all equal.
Keeping all the previous notations, for $i=1,...,m$ consider the linear systems $L(-a_iF)$. Apply now Theorem \ref{continuousresolution} to all of them and get

$$\xymatrix {{\widetilde X} \ar[r]^{\lambda}\ar[dr]^{\widetilde f} & X \ar[d]^{f} \\ & \mathbb{P}^1}$$

\noindent where $\lambda$ is  a composition of a blow-up and an \'etale Galois map induced by ${\widehat A}_d$ for certain $d$, verifying the following conditions

\begin{itemize}
\item For $i=1,...,m$ and for any $\alpha \in {\widehat A}$ we have $\lambda^*(L(-a_iF))\otimes \alpha=W_i+N_{\alpha,i}$.
\item $W_i$ does not depend on $\alpha$ and is the fixed component of the linear system.
\item $N_{\alpha,i}$ induce a base point free linear system and for all $\alpha$ we have $N_{\alpha,i}=N_{0,i}\otimes \alpha$.
\end{itemize}

Furthermore, since the kernel of the map ${\rm Pic}^0(X) \longrightarrow {\rm Pic}^0(F)$ is finite, we can choose $d$ in such a way that ${\widetilde F}={\lambda}^*F$ is a connected \'etale cover of $F$ and hence ${\widetilde f}$ is a fibration.

Let ${\widetilde L}=\lambda ^* L$, and $r={\rm deg}\mu$. By the choice of $L$ and projection formula we have

$${\widetilde {\cal E}}={\widetilde f}_*{\widetilde L}={\cal E}^{\oplus r}={\cal O}_{\mathbb{P}^1}(a_1)^{\oplus r}\oplus ...\oplus {\cal O}_{\mathbb{P}^1}(a_l)^{\oplus r}$$

\noindent Applying Xiao's method to ${\widetilde f}$ we have that the vector bundles ${\widetilde {\cal E}}_j$ induce base point free linear systems ${\widetilde N}_j$ on ${\widetilde X}$ and ${\widetilde P}_j$ on ${\widetilde F}$. Among these we have for $i=1,...,m$

$${\widetilde {\cal E}}_{ri}={\cal O}_{\mathbb{P}^1}(a_1)^{\oplus r}\oplus ...\oplus {\cal O}_{\mathbb{P}^1}(a_i)^{\oplus r}$$

\noindent which induce precisely the linear systems ${\widetilde N}_{ri}
=N_{0,i}$ defined above. Then ${\widetilde P}_{ri}=N_{0,i|{\widetilde F}}$ which by construction has a space of sections of dimension at least ${\rm rank}{\widetilde E}_{ri}=ri$. Following the previous conventions recall that if $r(i-1)<j\leq ri$ then ${\widetilde P}_j={\widetilde P}_{ri}$.

Summing up we obtain

\begin{proposition}\label{arreglalotodo}

\begin{itemize}
\item [(i)] $h^0_{\widetilde a}({\widetilde P}_{ri})\geq ri$.
\item [(ii)] If $X$ is of maximal $a$-dimension, we can choose $\lambda$ in such a way that for all $j=1,...,rm$ the linear systems $|{\widetilde P}
_j|$ are generically finite.
\end{itemize}
\end{proposition}

\begin{proof}
\noindent (i) We can do the same construction for ${\widetilde {\cal E}}_{\alpha}={\widetilde f}_*({\widetilde L}\otimes \alpha)$. Similarly we obtain linear systems ${\widetilde N}^{\alpha}_{ri}=N_{\alpha,i}$ such that when restricted to ${\widetilde F}$ they induce ${\widetilde P}^{\alpha}_{ri}$ of dimension at least $ri$. By construction  for any $\alpha$ we have that $N_{\alpha,i}=N_{0,i}\otimes \alpha$ and hence the same happens when restricting to ${\widetilde F}$. Hence for all $\alpha$ we have that $h^0({\widetilde F}, {\widetilde P}_{ri}\otimes \alpha)=h^0({\widetilde F},{\widetilde P}^{\alpha}_{ri})\geq ri$ and so $h^0_{\widetilde a}({\widetilde P}_{ri})\geq ri$.

\noindent (ii) We have that ${\widetilde P}_j={\widetilde P}_{ri}$ for some $i$. By Theorem 3.5 we can modify $\lambda$ by a multiplication map with $d>>0$ such that the maps induced by $N_{\alpha,i}$ are all generically finite.

\end{proof}

\subsection{The proof}\label{prueba}

\bigskip
We will use freely the notations of subsection \ref{subsectionxiao}. Observe that if $h^0_a(L)=0$ then the result is trivially true. From now on we will consider that $h^0_a(L)\neq 0$. We can also assume that the map $a^*: {\widehat A} \longrightarrow {\rm Pic}^0X$ is injective, and hence apply freely the results of Section 3 and Subsection 5.1. Indeed, consider the abelian variety $B={\rm Im}(a^*)$ and let $C={\widehat B}$. We have a factorization of $a$

 $$\xymatrix {X \ar[r]^{{\rm alb}_X} & {\rm Alb}_X \ar[r]^{\pi} & C \ar[r] & A } $$

\noindent with $(\pi \circ {\rm alb}_X)^*:{\widehat C}=B \longrightarrow {\rm Pic}^0X$ injective and $h^0_a(L)=h^0_{\pi \circ {\rm alb}_X}(L)$. Since the map $C \longrightarrow A$ is \'etale onto the image of $a$, clearly the $a$-dimension of $X$ coincides with the $(\pi \circ {\rm alb}_X)$-dimension of $X$.

Observe that the injectivity property is stable by restriction to any $M$, a big and nef smooth divisor on $X$, since ${\rm Pic}^0X \cong {\rm Pic}^0M$ if ${\rm dim} \, M\geq 2$ and ${\rm Pic}^0X \subseteq {\rm Pic}^0M$ if ${\rm dim} \, M=1$.

\medskip

\noindent (i) Assume that $X$ is of maximal $a$-dimension. We proceed by induction on $n={\rm dim} \, X$.

%By lemma \ref{etale} up to an \'etale base change, we can assume that the linear system $|L|$ induces a generically finite map on $X$.

\medskip

\noindent
{\underline {\bf Case $n=1$.}} Let $X$ be a smooth curve of genus $g\geq 1$ and $L$ a nef line bundle. If ${\rm deg} L \leq {\rm deg} K_X=2g-2$ and it is non-special, then ${\rm deg} L \geq 2 h^0(X,L)$ by Riemann-Roch theorem. If ${\rm deg} L=r(2g-2)$ with $r \in \mathbb{Q}, r>1$ then again by Riemann-Roch theorem we obtain ${\rm deg}L = \frac{2r}{2r-1}h^0(X,L)$. It remains the case of an special divisor. Take the \'etale cover of $C$ induced by ${\widehat A}_d$, say $\mu: {\widetilde C} \longrightarrow C$. Consider ${\widetilde L}=\mu ^* L$ which is an special line bundle on ${\widetilde C}$. Hence, we can apply Clifford's theorem and obtain

$$d^{2q}{\rm deg}L={\rm deg}{\widetilde L} \geq 2 h^0({\widetilde C},{\widetilde L})-2\geq 2d^{2q}h^0_{a}(L)-2$$

\noindent where $q={\rm dim} \, A$. Since this holds for all $d$ we obtain ${\rm deg}L\geq 2h^0_a(L)$.

\medskip

\noindent
{\underline {\bf Case $n\geq 2$.}} We assume now that for any irregular variety $X'$ of dimension at most $n-1$, with a nontrivial map $a': Y \longrightarrow A'$ to an abelian variety, and for any nef line bundle $L'$, inequality (i) holds. The key argument here is to use this induction hypothesis in Xiao's method on a suitable \'etale cover of $X$ to prove an inequality (Step 1) from which we can apply the main idea of Pardini's proof in \cite{Pardini} (Step 2).

\medskip

{\underline {\it Step 1.}{\it Claim.}} Given a base point free linear system $|M|$ on $X$, with $M^n>0$, we have
$$L^n+nL^{n-1}M\geq \delta(L)\,n!\,h^0_a(L).$$

\noindent {\it Proof of Claim.} The general member of $|M|$ is smooth and irreducible. Take two general smooth members $F, F' \in |{M}|$. Consider a blow up $\epsilon: Y\longrightarrow X$ in order to get a fibration $f: Y \longrightarrow \mathbb{P}^1$ induced by $F$ and $F'$. Since the formula we want to prove is invariant in the algebraic class of $L$, we change $L$ by $L\otimes \beta$ in such a way that Proposition 5.5 applies for $\epsilon^*L$. Moreover we can get that $h^0(X,L)=h^0_a(L)$ and that $h^0(F,L_{|F})=h^0_a(L_{|F})$.

Then we can apply the construction of Subsection 5.1  to $f:Y \longrightarrow \mathbb{P}^1$ and get

$$\xymatrix {A \ar[r]^{\mu} & A \\{\widetilde X} \ar[u]^{{\widetilde a}}\ar[r]^{\mu} & X \ar[u]^{a}\\{\widetilde Y}\ar[u]^{\widetilde {\epsilon}}\ar[r]^{\lambda}\ar[dr]_{{\widetilde f}} &Y\ar[u]^{\epsilon} \ar[d]^{f}& \\&{\mathbb{P}_1}&}$$

\noindent where $\mu$ is of degree $r=d^{2q}$ ($q={\rm dim} \, A$). Let ${\widetilde L}=\lambda^*(\epsilon^*L)$ and ${\widetilde F}=\lambda^*F$, which is an irreducible \'etale cover of $F$ of degree $r$. We can apply Proposition 5.6 (ii) and so all the induced maps are generically finite and hence by (\ref{xiaotruncadagenericamentefinito}) we obtain

$${\widetilde L}^n \geq n[{\widetilde P}_r^{n-1}(a_1-a_2)+...+{\widetilde P}_{rm}^{n-1}a_m].$$

Observe that ${\widetilde F}$ is $(n-1)$-dimensional, of maximal ${\widetilde a}_{|{\widetilde F}}$-dimension and the map ${\widetilde a}_{|{\widetilde F}}$ is non trivial. So we apply the induction hypothesis for the nef line bundles ${\widetilde P}_i$ on ${\widetilde F}$, and for all $i=1,...,m$ we get

$${\widetilde P}_{ri}^{n-1}\geq \delta ({\widetilde P}_{ri})\,(n-1)!\,h^0_{{\widetilde a}}({\widetilde P}_{ri})\geq \delta (L)\,(n-1)!\,h^0_{{\widetilde a}}({\widetilde P}_{ri})$$

\noindent the last inequality holding by Remark \ref{delta2}. By Proposition 5.6 (i) we have that

$$h^0_{{\widetilde a}}({\widetilde P}_{ri})\geq ri$$

Thus

$$rL^n={\widetilde L}^n\geq r\delta (L)\, n!\,(a_1+...+a_m)\geq r\delta (L)\, n!\,(h^0_{a}(L)-h^0_{a}({\epsilon^*L}_{|F}))$$

%Since ${\epsilon}^*L(-a_1F)$ is effective and $\epsilon^*L$ is nef we have that $(L)^n\geq a_1(L_{|F})^{n-1}$ and u

Using again induction for $\epsilon^*L_{|F}$ we have

$$(\epsilon^*L_{|F})^{n-1}\geq \delta (\epsilon^*L_{|F})\,(n-1)!\,h^0_{a}(\epsilon^*L_{|F})\geq \delta (L)\,(n-1)!\,h^0_{a}(\epsilon^*L_{|F})$$

\noindent and summing up, since $\epsilon_*F=M$ we finally obtain

$$L^n+ n\,L^{n-1}M \geq\delta(L) \,n! \,h^0_a(L).$$

{\underline {\it Step 2.} Let us apply now Pardini's covering trick to prove the statement. Consider again $d\in \mathbb{N}$ and the \'etale Galois map induced by multiplication

$$\xymatrix {A \ar[r]^{\mu} & A\\{\widetilde X} \ar[u]^{{\widetilde a}}\ar[r]^{\mu} & X \ar[u]^{a}}$$

Let $H$ be a fixed very ample line bundle on $A$ and let $M=a^*H$, ${\widetilde M}={\widetilde a}^*(H)$. By \cite{LB} Ch2. Prop. 3.5 we have

$${\mu}^*H \equiv d^2 H$$

\noindent and so
\begin{equation}\label{dcuadrado}
{\widetilde M} \equiv \frac{1}{d^2}{\mu}^*M.
\end{equation}

Define again

$${\widetilde L}:={\mu}^*L$$

We have

\begin{itemize}
\item  $h^0_{\widetilde a}({\widetilde L})=d^{2q}h^0_a(L)$ by Proposition 2.8, and $\delta ({\widetilde L}) = \delta (L)$ by Remark \ref{delta2}.
\item  For all $i=0,...,n$  ${\widetilde L}^{n-i}{\widetilde M}^i=d^{2q-2i}L^{n-i}M^i$ by (\ref{dcuadrado}).
\end{itemize}

We can apply now the Claim of Step 1 to $({\widetilde M},{\widetilde X}, {\widetilde L})$:

$$({\widetilde L})^n+n({\widetilde L})^{n-1}{\widetilde M}\geq \delta({\widetilde L})\, n!\, h^0_a({\widetilde L})$$

And hence:

$$d^{2q}L^n+nd^{2q-2}L^{n-1}M\geq d^{2q}\delta(L)\, n!\, h^0_a(L)$$

\noindent which holds for all $d$. Thus we can conclude

$$L^n \geq \delta (L)\, n!\, h^0_a(L).$$

\medskip

\noindent (ii) Assume now that $1 \leq k={\rm dim} \, a(X) < n$ and that $M
_{|G}$ is big, $G$ being  a general fibre of the algebraic fibre space induced by $a$ and $M$ the continuous moving part of the linear system $|L|$.

Following Theorem \ref{etale} (ii), up to a composition of a blow-up and an \'etale cover, we have $\lambda ^*L=W+N$ with $|N|$ base point free and generically finite, $h^0_a(N)=({\rm deg}\mu)h^0_a(L)$ and $\delta(N) \geq \delta (L)$ (see Remark \ref{delta2}). Since $L$ and $N$ are nef we have $({\rm deg}\mu) L^n=(\lambda ^*L)^n \geq N^n$. Hence its enough to prove the statement for $N$.

Take general elements $N_1,...N_{n-k}\in |N|$ and let $T=N_1\cap...\cap N_{n-k}$. We have that $T$ is smooth, $k$-dimensional and dominates $a(X)$ since $N$ is transversal to a general $G$ (it induces a generically finite map). Hence $T$ is of maximal $a$-dimension. Let $N_T=N_{|T}$. We have $h^0_a(N_T)\geq h^0_a(N)$ (by Proposition 2.8 (iii)) and $\delta(N_T) \geq \delta (N)$ (by Remark \ref{delta2}). Then we apply (i) to the pair $(T,N_T)$ with respect to $a$ and obtain

$$N^n=(N_T)^k \geq \delta(N_T) \, k! \, h^0_a(N_T)\geq \delta(N) \, k! \, h^0_a(N).$$

\medskip

\noindent (iii) As in (ii) up to an \'etale cover and blow-up we obtain $|\lambda^*L|=W+|N|$. But in this case $N=\phi^*R$ where $\phi: {\widetilde X} \longrightarrow Z$ is the algebraic fibre space induced by $|N|$, ${\rm dim} \, Z\geq k$ and the linear system $|R|$ on $Z$ is generically finite. Up to blow-ups on ${\widetilde X}$ and $Z$ we can assume that $Z$ is smooth. $L$ is big so also $\lambda^*L$ is. Contrary to (ii), in this case there is no clear relation between $\delta(R)$ and $\delta(N)\geq \delta(\lambda^*L)\geq\delta(L)$, so we only use $\delta(R)\geq 1$.

Let $r={\rm dim} \, Z$ and let ${\overline G}$ be the fibre of $\phi$. By nefness of $\lambda^*L$ and $N$ and bigness of $(\lambda^*L)_{|G}$ we have

$$L^n=(\lambda^*L)^n\geq (\lambda^*L)^{n-r}\phi^*(R)^r=((\lambda^*L)_{{\overline G}})^{n-r}R^r\geq R^r.$$

\noindent If ${\rm dim} \, Z=k$ we just apply (i) to the pair $(Z,R)$. If ${\rm dim} \, Z>k$, since $|R|$ is base point free on $Z$ and induces a generically finite map, we can apply to it the same argument as in (ii). In any case we obtain

$$R^r\geq \delta(R)\,k!\,h^0_{\widetilde a}(R)\geq k! \,h^0_{a'}(R)=k! \, h^0_a(L).$$

\begin{remark} As a corollary of proof, observe that in the proof of (iii) we can obtain the same inequality as in (ii) provided that $((\lambda^*L)
_{{\overline G}})^{n-r}\geq 2$. In other words, if (ii) does not hold, then ${\widetilde X}$ is fibred by a family of varieties ${\overline G}$ with $((\lambda^*L)_{{\overline G}})^{n-r}=1$.

When $X$ is Gorenstein and minimal, $L=K_X$ and $a=alb_X$, observe that this implies that the ${\widetilde a}$-fibres of ${\widetilde X}$ are fibred by ${\overline G}$ such that $K_{\overline G}^{n-r}=1$. Since the general fibre of the algebraic fibre space induced by ${\widetilde a}$ is isomorphic to the general fibre of the algebraic fibre space induced by $alb_X$ on $X$, then we can conclude that if $K^n_X < 2k!\,\chi(\omega_X)$, then $alb_X$ factors through a fibration with general fibre $G$ such that $K^{{\rm dim}G}_G=1$. Observe that this cannot happen when ${\rm dim} \, alb(X)=n-1$.

%In particular, either $G$ is regular or $\chi(\omega_G)\leq 1$.
\end{remark}

\addcontentsline{toc}{section}{References}

\bigskip
\bigskip
\noindent Miguel \'Angel Barja \\  Departament de Matem\`atica  Aplicada I, ETSEIB-Facultat de Matem\`atiques i Estad{\'\i}stica \\
Universitat Polit\`ecnica de Catalunya-BarcelonaTECH \\ Avda. Diagonal, 647 \\ 08028 Barcelona (Spain).\\
e-mail: \textsl{miguel.angel.barja@upc.edu}

\end{document}